\documentclass[a4paper,11pt]{article}
\usepackage{latexsym, amssymb, amsmath, amsthm, bbm}
\usepackage[all,cmtip]{xy}
\usepackage{amsmath}
\usepackage{amsthm}
\usepackage{amssymb}
\usepackage{amsfonts}
\usepackage{amsrefs}
\usepackage{color,authblk}
\usepackage{tikz}

\numberwithin{equation}{subsection}
\newtheorem{theorem}{Theorem}[section]
\newtheorem{lemma}[theorem]{Lemma}
\newtheorem{example}[theorem]{Example}
\newtheorem{proposition}[theorem]{Proposition}
\newtheorem{corollary}[theorem]{Corollary}
\newtheorem{remark}[theorem]{Remark}
\newtheorem{definition}[theorem]{Definition}

\begin{document}

\title{Morita invariance of equivariant and Hopf-cyclic cohomology of module algebras over Hopf algebroids}
\author{Mamta Balodi}

\affil{\small{Department of Mathematics,\\ Indian Institute of Science, Bengaluru 560012, INDIA\\  mamtabalodi@iisc.ac.in}}
\date{}
\maketitle

\begin{abstract}
We investigate the equivariant and Hopf-cyclic cohomology of module algebras over Hopf algebroids and derive their Morita invariance. For this, we use the tools developed by McCarthy for $k$-linear categories and subsequently by Kaygun and Khalkhali for a certain class of categories, called as $H$-categories, for a Hopf algebra $H$.
\end{abstract}

\section{Introduction}
Hopf algebroids are one of the various generalizations, see \cites{S2,xu,Lu,bs}, of Hopf algebras over non-commutative base algebras. The one in which we are interested is introduced by B\"ohm and Szlach\'anyi in \cite{bs}. A Hopf algebroid consists of two compatible bialgebroid structures, called left and right bialgebroids (see \cite{ks}), on a given algebra, which are related to each other with the antipode. Hopf algebras, weak Hopf algebras \cite{bns} and Connes-Moscovici algebras \cite{CM2} are some examples of Hopf algebroids. Since the study of Hopf algebroids is very recent, there are many results on Hopf algebras yet to be investigated for Hopf algebroids. This paper is an attempt towards understanding the Hopf-cyclic theory of Hopf algebroids.

The cyclic (co)homology of Hopf algebras was discovered by Connes and Moscovici \cite{CM1}. A general framework was provided in \cite{HKRS} for cyclic theories whose cyclicity is based on Hopf-algebraic structures. These are referred to as Hopf-cyclic theories. The Hopf-cyclic (co)homology provides a natural generalization of Lie algebra (co)homology to the general framework of Noncommutative Geometry. 

McCarthy in \cite{carthy} defined the Hochschild and cyclic homology groups for a $k$-linear exact category $\mathcal{C}$ using the additive cyclic nerve of $\mathcal{C}$. He also defined the notion of a special homotopy which is a sufficient condition for two maps to induce the same maps of Hochschild and cyclic homology groups. As an application, he improved the proof of Morita invariance of cyclic homology given by him in \cite{carthy2}. Subsequently Kaygun and Khalkhali defined the notion of an $H$-category in \cite{HKRS} for a Hopf algebra $H$ and developed a cyclic (co)homology theory for $H$-categories. It was shown there that the cyclic $H$-module associated with a module algebra can also be defined as the cyclic module of a certain $H$-category using the nerve. They also proved that equivalent $H$-categories have isomorphic cyclic (co)homologies which then provided the Morita invariance of Hopf-cyclic (co)homology theory. This paper arose from the desire to generalize these results in case of Hopf algebroids. 


The paper is organized as follows: In Section 2, we recall basic definitions and a brief summary of the properties of Hopf algebroids and their modules and comodules. In Section 3, we first define the equivariant and Hopf-cyclic cohomology for a module algebra $A$ over a Hopf algebroid $\mathcal{H}$, with coefficients in a right-right module/comodule $M$. Then we extend the notion of an $\mathcal{H}$-category for Hopf algebroids defined for Hopf algebras in \cite{kk}, and prove that the category of all finitely generated left $\mathcal{H}$-equivariant right $A$-modules with the morphisms of right $A$-modules, is an $\mathcal{H}$-category. Then, analogous to \cite{kk}, we define a cyclic right $\mathcal{H}$-module $\mathcal{Q}_\bullet(\mathcal{C},M)$ for an $\mathcal{H}$-category and stable module/comodule $M$. Then we recall the notion of special homotopy \cite{carthy}, and, by adopting the categorical approach developed in \cite{kk}, we prove that two $\mathcal{H}$-equivariant equivalent functors induce a special homotopy between the cyclic right $\mathcal{H}$-modules as defined above. As a corollary, we get that equivalent $\mathcal{H}$-categories have isomorphic cyclic cohomologies. With the help of these ingredients, we derive Morita invariance of Hopf-cyclic cohomology for module algebras over Hopf algebroids.

Throughout the paper, $k$ will denote a commutative ring.

\section{Preliminaries}
\subsection{Bialgebroids}
Bialgebroids are generalization of bialgebras over non-commutative base algebras \cites{ks,Lu,S}. One works with rings and corings over non-commutative base algebras, instead of algebras and coalgebras over commutative rings. Let us recall these notions briefly from \cite{b}. 

Let $L$ be a $k$-algebra. An $L$-ring is a triple $(A,\mu,\eta)$, where $A$ is an $L$-bimodule and $\mu:A \otimes_L A \longrightarrow A$ and $\eta:L \longrightarrow A$ are $L$-bimodule maps satisfying the associativity and unit conditions 
$$\mu(\mu \otimes_L \text{id})=\mu(\text{id} \otimes_L \mu), \quad \mu(\eta \otimes_L \text{id})=\text{id}=\mu(\text{id} \otimes_L \eta).$$
An $L$-coring is a triple $(C,\Delta,\varepsilon)$, where $C$ is an $L$-bimodule and $\Delta:C \longrightarrow C \otimes_L C$ and $\varepsilon:C \longrightarrow L$ are $L$-bimodule maps satisfying the coassociativity and counit conditions 
$$(\Delta \otimes_L \text{id})\Delta=(\text{id} \otimes_L \Delta)\Delta, \quad (\varepsilon \otimes_L \text{id})\Delta=\text{id}=(\text{id} \otimes_L \varepsilon)\Delta.$$
A left module over an $L$-ring $(A,\mu,\eta)$ is a pair $(M,\phi)$, where $M$ is a left $L$-module and $\phi:A \otimes_L M \longrightarrow M$ is an $L$-module morphism such that
$$\phi(\mu \otimes_L \text{id})=\phi(\text{id} \otimes_L \phi), \quad \phi(\eta \otimes_L \text{id})=\text{id}.$$
A left $A$-module morphism $f:M \longrightarrow N$ is a left $R$-module map satisfying $f\phi_M=\phi_N(\text{id} \otimes_R f).$ Symmetrically, one can define right $L$-modules.

\begin{definition}
A left $L$-bialgebroid consists of the data $\mathcal{H}_L=(H,L,s_L, t_L, \Delta_L,\varepsilon_L)$, where $H$ and $L$ are $k$-algebras, called the total and base algebras, respectively. $H$ is an $L \otimes_k L^{\text{op}}$-ring via the algebra homomorphisms $s_L:L \longrightarrow H $ and $t_L:L^{\text{op}} \longrightarrow H$ called as source and target maps, respectively. This means that their images are commuting subalgebras of $H$, and thus making $H$ an $L$-$L$ bimodule via
\begin{equation}\label{lbimod}
l \cdot h \cdot l'= s_L(l)t_L(l')h, \quad l,l' \in L, h \in H.
\end{equation} 
The triple $(H, \Delta_L, \varepsilon_L)$ is an $L$-coring such that
\begin{gather*}
h_{(1)}t_L(l) \otimes_L h_{(2)} =  h_{(1)} \otimes_L h_{(2)}s_L(l),\\
\Delta_L(1_H)=1_H \otimes_L 1_H,\\
\Delta_L(hg)=\Delta_L(h)\Delta_L(g),\\
\varepsilon_L(1_H)=1_L,\\
\varepsilon_L(hs_L \varepsilon_L(g))=\varepsilon_L(hg)=\varepsilon_L(ht_L \varepsilon_L(g)),
\end{gather*}
for any $l \in L$, $h,g \in H$. Here we use the Sweedler's notation $\Delta_L(h)=h_{(1)} \otimes_L h_{(2)}$, with summation symbol suppressed.

Similarly, a right $R$-bialgebroid consists of the data $\mathcal{H}_R=(H,R,s_R, t_R, \Delta_R,\varepsilon_R)$, where $H$ and $R$ are $k$-algebras, called the total and base algebras, respectively. $H$ is an $R \otimes_k R^{\text{op}}$-ring via the algebra homomorphisms $s_R:R \longrightarrow H $ (source) and $t_R:R^{\text{op}} \longrightarrow H$ (target), i.e. their images are commuting subalgebras of $H$, and thus making $H$ an $R$-$R$ bimodule via
\begin{equation}\label{rbimod}
r \cdot h \cdot r'= hs_R(r)t_R(r'), \quad r,r' \in R, h \in H.
\end{equation} 
The triple $(H, \Delta_R, \varepsilon_R)$ is an $R$-coring such that
\begin{gather*}
s_R(r)h^{(1)} \otimes_R h^{(2)} =  h^{(1)} \otimes_R t_R(r)h^{(2)},\\
\Delta_R(1_H)=1_H \otimes_R 1_H,\\
\Delta_R(hg)=\Delta_R(h)\Delta_R(g),\\
\varepsilon_R(1_H)=1_R,\\
\varepsilon_R(s_R \varepsilon_R(h)g)=\varepsilon_R(hg)=\varepsilon_R(t_R \varepsilon_R(h)g),
\end{gather*}
for any $r \in R$, $h,g \in H$. Here we use the Sweedler's notation $\Delta_R(h)=h^{(1)} \otimes_R h^{(2)}$, with summation symbol suppressed.
\end{definition}

\begin{example}\label{trivbia}
A bialgebra $B$ over a commutative ring $k$ determines a left (right) $k$-bialgebroid in which both the source and target maps are equal to the unit map $k \longrightarrow B$.
\end{example}

\subsection{Hopf algebroids}
\begin{definition}
A Hopf algebroid $\mathcal{H}=(\mathcal{H}_L,\mathcal{H}_R,S)$ over base algebras $L$ and $R$ consists of a left bialgebroid $\mathcal{H}_L=(H,L,s_L, t_L, \Delta_L,\varepsilon_L)$, a right bialgebroid $\mathcal{H}_R=(H,R,s_R, t_R, \Delta_R,\varepsilon_R)$ on the same total algebra $H$, and a $k$-linear map $S:H \longrightarrow H$, called the antipode, such that following axioms hold:
\begin{gather}
s_L \varepsilon_L t_R = t_R, \quad t_L \varepsilon_L s_R = s_R, \quad s_R \varepsilon_R t_L = t_L, \quad t_R \varepsilon_R s_L = s_L,\label{anti-iso}\\
(\Delta_L \otimes_R \text{id})\Delta_R= (\text{id} \otimes_L \Delta_R)\Delta_L, \quad (\Delta_R \otimes_L \text{id})\Delta_L= (\text{id} \otimes_R \Delta_L)\Delta_R,\\
S(t_L(l)ht_R(r)) = s_R(r)S(h)s_L(l),\\
S(h_{(1)})h_{(2)}=s_R\varepsilon_R(h), \quad h^{(1)} S(h^{(2)})=s_L\varepsilon_L(h),\label{antipode}
\end{gather}
for any $l \in L$, $r \in R$ and $h \in H$.
\end{definition}
By axiom (\ref{anti-iso}), it follows that the base algebras $L$ and $R$ are anti-isomorphic. In fact, we have inverse algebra isomorphisms (see \cite{b})
\begin{align}
\varepsilon_L\circ s_R:R^{\text{op}} \longrightarrow L \quad \text{and} \quad \varepsilon_R\circ t_L:L \longrightarrow R^{\text{op}},\label{invalgiso}\\
\varepsilon_R\circ s_L:L^{\text{op}} \longrightarrow R \quad \text{and} \quad \varepsilon_L\circ t_R:R \longrightarrow L^{\text{op}}.
\end{align}

Let us look at some examples of Hopf algebroids before having some more properties of the antipode.
\begin{example}
A Hopf algebra $(H,\mu,\eta,\Delta,\varepsilon,S)$ over a commutative ring $k$ is a Hopf algebroid with $L=R=k$, $s_L=t_L=s_R=t_R=\eta:k \longrightarrow H$, $\varepsilon_L=\varepsilon_R=\varepsilon:H \longrightarrow k$, and $\Delta_L=\Delta_R=\Delta:H \longrightarrow H \otimes_k H.$
\end{example}

\begin{example}
A Hopf algebra $(H,\mu,\eta,\Delta,\varepsilon,S)$ with a character $\delta:H \longrightarrow k$ and the twisted antipode, defined by Connes and Moscovici in \cite{CM2}, $\tilde{S}=\mu(\delta \otimes S)\Delta$ determines a Hopf algebroid over $k$ as follows: the underlying $k$-bialgebra $(H,\mu,\eta,\Delta,\varepsilon)$ is a left $k$-bialgebroid as in Example \ref{trivbia}. For the right $k$-bialgebroid structure, again the source and target maps are given by $\eta$ and the right coproduct $\Delta_r: H \longrightarrow H \otimes_k H$ and the right counit $\varepsilon_R:H \longrightarrow k$ are given as
$$\Delta_r(h)=h_{(1)} \otimes_k \eta(\delta S(h_{(2)}))h_{(3)}, \quad \varepsilon_R:=\delta:H \longrightarrow k.$$ 
\end{example}

Let us recall some properties of the antipode $S$ of a Hopf algebroid $\mathcal{H}=(\mathcal{H}_L,\mathcal{H}_R,S)$ from \cite{b2} which will be used later. The antipode $S$ is an anti-algebra map on the total algebra $H$ i.e. for any $h, g \in H$,
\begin{equation}\label{antialg}
S(1_H)=1_H, \quad S(gh)=S(h)S(g).
\end{equation}
The antipode $S$ is also a anti-coring map $\mathcal{H}_L \longrightarrow \mathcal{H}_R$ and $\mathcal{H}_R \longrightarrow \mathcal{H}_L$ which means that the following hold:
\begin{gather}
\varepsilon_R S=\varepsilon_Rs_L\varepsilon_L, \quad (S \otimes_R \text{id})\Delta_R=(S \otimes_R \text{id})(\tau_L \circ \Delta_L),\\
\varepsilon_L S=\varepsilon_Ls_R\varepsilon_R, \quad (S \otimes_L \text{id})\Delta_L=(S \otimes_L \text{id})(\tau_R \circ \Delta_R),
\end{gather}
where $\tau_I:\mathcal{H}_I \otimes_I \mathcal{H}_I \longrightarrow \mathcal{H}_I \otimes_I \mathcal{H}_I$ is the flip map ($h \otimes g \mapsto g \otimes h$) for $I \in \{L,R\}$.

Further, if the antipode $S$ is bijective, then $S^{-1}$ is also an anti-algebra map on the total algebra $H$, as well as an anti-coring map from $\mathcal{H}_L \longrightarrow \mathcal{H}_R$ and $\mathcal{H}_R \longrightarrow \mathcal{H}_L$. Moreover, for any $h \in H$, the following hold:
\begin{align}
S^{-1}(h_{(2)})h_{(1)}=t_R\varepsilon_R(h),& \quad h^{(2)} S^{-1}(h^{(1)})=t_L\varepsilon_L(h),\label{invantipode}\\
t_R=S^{-1}s_R, \quad t_L\varepsilon_Lt_R=S^{-1}t_R,& \quad  t_L=S^{-1}s_L, \quad t_R\varepsilon_Rt_L=S^{-1}t_L,\\
\varepsilon_Lt_R\varepsilon_R=\varepsilon_LS^{-1},& \quad \varepsilon_Rt_L\varepsilon_L=\varepsilon_RS^{-1}.
\end{align}
%
%

\subsection{Modules and comodules for Hopf algebroids}
Let $\mathcal{H}_L=(H,L,s_L,t_L,\Delta_L,\varepsilon_L)$ be a left bialgebroid. A left $\mathcal{H}_L$-module means a left module of the $L \otimes_k L^{\text{op}}$-ring $(H,\mu,\eta)$. A left $\mathcal{H}_L$-module $M$ carries an $L$-$L$ bimodule structure $$l \cdot m \cdot l'=s_L(l)t_L(l')m,$$
for $l,l' \in L$ and $m \in M$. For a Hopf algebroid $\mathcal{H}=(\mathcal{H}_L,\mathcal{H}_R,S)$, a left $\mathcal{H}$-module is just a left $\mathcal{H}_L$-module and a right $\mathcal{H}$-module is just a right 
$\mathcal{H}_R$-module. A left $\mathcal{H}_L$-module morphism is a morphism of left modules of the $L \otimes_k L^{\text{op}}$-ring $(H,\mu,\eta)$. A right $\mathcal{H}_R$-module morphism can be defined symmetrically. 

\begin{definition}
Let $\mathcal{H}_L=(H,L,s_L,t_L,\Delta_L,\varepsilon_L)$ be a left bialgebroid. A $k$-algebra $A$ is called a left $\mathcal{H}_L$-module algebra, if $A$ is a left $\mathcal{H}_L$-module inheriting an $L$-$L$ bimodule structure from the $\mathcal{H}_L$-action denoted by $\triangleright$: $$l \cdot a \cdot l'=(l \cdot 1_H \cdot l')a=s_L(l)t_L(l')\triangleright a;$$ such that
\begin{gather*}
h \triangleright (ab)=(h_{(1)} \triangleright a)(h_{(2)} \triangleright b),\\ h \triangleright 1_A=\varepsilon_L(h) \cdot 1_A = 1_A \cdot \varepsilon_L(h) \equiv s_L(\varepsilon_L(h))\triangleright 1_A \equiv t_L(\varepsilon_L(h)) \triangleright 1_A.
\end{gather*}
for $l,l' \in L$, $a,b \in A$ and $h \in H$. A left $\mathcal{H}$-module algebra over a Hopf algebroid $\mathcal{H}=(\mathcal{H}_L,\mathcal{H}_R,S)$ means a left $\mathcal{H}_L$-module algebra. Similarly one can define a right $\mathcal{H}$-module algebra. 
\end{definition}

\begin{definition}
Let $\mathcal{H}=(\mathcal{H}_L,\mathcal{H}_R,S)$ be a Hopf algebroid. If $M$ is a left $\mathcal{H}$-module as well as a right $A$-module  such that $$h(ma)=(h_{(1)} \triangleright m)(h_{(2)} a),$$
then $M$ is called a left-right $\mathcal{H}$-equivariant $A$-module. 
\end{definition}

\begin{example}
A left module algebra $A$ over a Hopf algebroid $\mathcal{H}$ is an left-right $\mathcal{H}$-equivariant $A$-module, where $A$ is a right module over itself via its multiplication.
\end{example}

\begin{definition}
A right comodule over a Hopf algebroid $\mathcal{H}=(\mathcal{H}_L,\mathcal{H}_R,S)$ is a triple $(M, \rho_L,\rho_R)$, where $(M, \rho_L)$ is a right comodule over the $L$-coring $(\mathcal{H}_L,\Delta_L, \varepsilon_L)$ and $(M, \rho_R)$ is a right comodule over the $R$-coring $(\mathcal{H}_R,\Delta_R, \varepsilon_R)$ such that the $L$-$L$ and $R$-$R$ bimodule structures of $M$ satisfy
$$l \cdot m \cdot l'=\varepsilon_R t_L(l') \cdot m \cdot \varepsilon_R t_L(l),$$
and the compatibility conditions hold true:
\begin{gather}
(\rho_R \otimes_L H)\rho_L=(M \otimes_R \Delta_L)\rho_R,\\
(\rho_L \otimes_R H)\rho_R=(M \otimes_L \Delta_R)\rho_L.\label{comodcomp}
\end{gather}
\end{definition}

We use the notation $\rho_R(m)=m^{[0]} \otimes_R m^{[1]}$ and $\rho_L(m)=m_{[0]} \otimes_L m_{[1]}$ with summation symbol omitted. Symmetrically, one can define a left comodule over a Hopf algebroid $\mathcal{H}=(\mathcal{H}_L,\mathcal{H}_R,S)$. A  right-right (left-left) $\mathcal{H}$-module/comodule $M$ means a right (left) $\mathcal{H}$-module which is also a right (left) $\mathcal{H}$-comodule with no assumption on compatibilty between action and coaction. Moreover, $M$ is said to be stable if $m^{[0]}m^{[1]}=m=m_{[0]}m_{[1]}$ for all $m \in M$.


\section{Morita invariance}
In this section, we define the $\mathcal{H}$-equivariant and Hopf-cyclic cohomology of $\mathcal{H}$-module algebras over a Hopf algebroid $\mathcal{H}$ with coefficients in a stable $\mathcal{H}$-module/comodule. We adopt the categorical approach, developed by Kaygun and Khalkhali in \cite{kk}, for certain class of categories called as $\mathcal{H}$-categories. Using this approach and the notion of special homotopy, defined by McCarthy in \cite{carthy}, for cyclic $\mathcal{H}$-modules, we derive the Morita invariance of Hopf-cyclic cohomology for Hopf algebroids.

\subsection{$\mathcal{H}$-equivariant and Hopf-cyclic cohomology of $\mathcal{H}$-module algebras}\label{EHC}
Let $\mathcal{H}=(\mathcal{H}_L,\mathcal{H}_R,S)$ be a Hopf algebroid. A cyclic $R$-$R$-bimodule $X_\bullet$ is called a cyclic right $\mathcal{H}$-module if for each $n \geq 0$, the module $X_n$ is a right $\mathcal{H}$-module and all the structure morphisms are $\mathcal{H}$-module morphisms. 
A para-cyclic $\mathcal{H}$-module is the same as a cyclic $\mathcal{H}$-module except that the cyclic operators $\tau_n$ on $X_n$ do not necessarily satisfy $\tau^{n+1}_n = \text{id}$ for $n \geq 0$, and a pseudo para-cyclic $\mathcal{H}$-module is the same as a para-cyclic $\mathcal{H}$-module except that the face map $d_n^{n-1}:X_n \longrightarrow X_{n-1}$ and $\tau_n$ need not be $\mathcal{H}$-module morphims.

Let $A$ be a right $\mathcal{H}$-module algebra and $M$ be a stable right-right $\mathcal{H}$-module/comodule. We define a simplicial $R$-module as well as a right $\mathcal{H}$-module $\mathcal{T}_\bullet(A,M)$ by setting $\mathcal{T}_n(A,M):=A^{\otimes n+1} \otimes_R M$. The structure morphisms are given by

$\begin{array}{lll}
\vspace{.2cm}
d_i^{n-1}(a_0 \otimes_R \cdots \otimes_R a_n \otimes_R m) &=& \begin{cases} a_0 \otimes_R \cdots \otimes_R a_i a_{i+1} \otimes_R \cdots \otimes_R a_n \otimes_R m \quad \hspace{.4cm} \text{if}~  0 \leq i \leq n-1\\
(a_nm^{[1]})a_0 \otimes_R a_1 \otimes_R \cdots \otimes_R a_{n-1} \otimes_L m^{[0]} \quad \text{if}~ i=n\end{cases}\\
\vspace{.2cm}
s_i^n(a_0 \otimes_R \cdots \otimes_R a_n \otimes_R m) &=& a_0 \otimes_R \cdots \otimes_R a_i \otimes_R 1_A \otimes_R \cdots \otimes_R a_n \otimes_R m \\
\vspace{.5cm}
\tau_n(a_0 \otimes_R \cdots \otimes_R a_n \otimes_R m) &=& a_n m^{[1]} \otimes_R a_0 \otimes_R \cdots \otimes_R a_{n-1} \otimes_R m^{[0]},
\end{array}$

\noindent and the right $\mathcal{H}$-action is given by
$$(a_0 \otimes_R \cdots \otimes_R a_n \otimes_R m)h=a_0h^{(2)} \otimes_R \cdots \otimes_R a_nh^{(n+2)} \otimes_R mh^{(1)},$$
for any $n \geq 0$, $0 \leq i \leq n$, $a_0 \otimes_R \cdots \otimes_R a_n \otimes_R m \in \mathcal{T}_n(A,M)$ and $h \in \mathcal{H}$.

\begin{proposition}
$\mathcal{T}_\bullet(A,M)$ is a para-cyclic $R$-$R$-bimodule.
\end{proposition}
\begin{proof}
We first show that $\tau_n$ is invertible with inverse given by
$$\tau_n^{-1}(a_0 \otimes_R \cdots \otimes_R a_n \otimes_R m)=a_1 \otimes_R \cdots \otimes_R a_n \otimes_R a_0S^{-1}(m_{[1]}) \otimes_R m_{[0]}.$$
\begin{align*}
\tau_n^{-1}\tau_n(a_0 \otimes_R \cdots \otimes_R a_n \otimes_R m)&=a_0 \otimes_R \cdots \otimes_R a_{n-1} \otimes_R a_nm^{[1]}S^{-1}(m_{[1]}^{[0]}) \otimes_R m_{[0]}^{[0]}\\
\overset{(\ref{comodcomp})}{=} &a_0 \otimes_R \cdots \otimes_R a_{n-1} \otimes_R a_nm^{(2)}_{[1]}S^{-1}(m_{[1]}^{(1)}) \otimes_R m_{[0]}\\
\overset{(\ref{invantipode})}{=}&a_0 \otimes_R \cdots \otimes_R a_{n-1} \otimes_R a_nt_L\varepsilon_L(m_{[1]})\otimes_R m_{[0]}\\
= \hspace{.35cm} &a_0 \otimes_R \cdots \otimes_R a_{n-1} \otimes_R a_n\cdot  \varepsilon_Rt_L\varepsilon_L(m_{[1]})\otimes_R m_{[0]}\\
=\hspace{.35cm} &a_0 \otimes_R \cdots \otimes_R a_{n-1} \otimes_R a_n\otimes_R m_{[0]}\cdot  \varepsilon_Rt_L\varepsilon_L(m_{[1]})\\
\overset{(\ref{invalgiso})}{=}&a_0 \otimes_R \cdots \otimes_R a_{n-1} \otimes_R a_n\otimes_R m_{[0]}\cdot \varepsilon_L(m_{[1]})\\
=\hspace{.35cm} &a_0 \otimes_R \cdots \otimes_R a_{n-1} \otimes_R a_n\otimes_R m.
\end{align*}
The fourth and the last equality use the fact, respectively, that $A$ is a right $\mathcal{H}$-module algebra and $M$ is a right comodule over the $L$-coring $\mathcal{H}$. Similarly one can prove that $\tau_n\tau_n^{-1}=\text{id}$. Then the result follows by observing that the face and degeneracy maps can be rewritten as
$$d_i=\tau_{n-1}^id_0\tau_n^{-i}, \quad s_i=\tau_{n+1}^is_0\tau_n^{-i}. \qedhere$$ 
\end{proof}

Now for any $j \in \mathbb{N}$, we define a morphism $[\tau_*^j, \rho_n]$ from $\mathcal{T}_n(A,M) \otimes_R \mathcal{H}$ to $\mathcal{T}_n(A,M)$, where $\rho_n$ is the $\mathcal{H}$-module structure morphism on $\mathcal{T}_n(A,M)$:
$$[\tau_*^j, \rho_n]=\tau_n^j\rho_n-\rho_n(\tau_n^j \otimes \text{id}).$$
Let $\mathcal{Q}^H_\bullet(A,M)$ be the para-cyclic $R$-module and graded $\mathcal{H}$-module defined by $$\mathcal{Q}^H_n(A,M):=\mathcal{T}_n(A,M)/\bigcap_{j \in \mathbb{Z}} \text{im}[\tau_*^j, \rho_n].$$ 
Next we define a graded $\mathcal{H}$-submodule $\mathcal{Q}_\bullet(A,M)$ of $\mathcal{Q}^H_\bullet(A,M)$ containing elements of the form $a_0 \otimes_R \cdots \otimes_R a_n \otimes_R m \in \mathcal{Q}^H_\bullet(A,M)$ such that
\begin{align*}
(a_0 \otimes_R \cdots \otimes_R a_n \otimes_R m)h &=(a_0 \otimes_R \cdots \otimes_R a_n \otimes_R m)\varepsilon_R(h).
\end{align*}

\begin{proposition}
$\mathcal{Q}_\bullet(A,M)$ is a cyclic right $\mathcal{H}$-module.
\end{proposition}
\begin{proof}
We check the identity $\tau_n^{n+1}=\text{id}$ on $\mathcal{Q}_\bullet(A,M)$. Others are similar to check.
\begin{align*}
\tau_n^{n+1}(a_0 \otimes_R \cdots \otimes_R a_n \otimes_R m)&=a_0 m^{[1]} \otimes_R a_1 \otimes_R \cdots \otimes_R a_nm^{[n+1]} \otimes_R m^{[0]},\\
&=a_0 m^{[2]} \otimes_R a_1 \otimes_R \cdots \otimes_R a_nm^{[n+2]} \otimes_R m^{(0)}m^{[1]},\\
&=(a_0 \otimes_R \cdots \otimes_R a_n \otimes_R m^{[0]})m^{[1]},\\
&=(a_0 \otimes_R \cdots \otimes_R a_n \otimes_R m^{[0]})\varepsilon_R(m^{[1]}),\\
&=(a_0 \otimes_R \cdots \otimes_R a_n \otimes_R m^{[0]}\varepsilon_R(m^{[1]})),\\
&=a_0 \otimes_R \cdots \otimes_R a_n \otimes_R m. \qedhere
\end{align*} 
\end{proof}

\begin{definition}
The cyclic $\mathcal{H}$-module $\mathcal{Q}_\bullet(A,M)$ is called the $\mathcal{H}$-equivariant cyclic-module of the pair $(A, M)$. The cyclic cohomology of $\mathcal{Q}_\bullet(A,M)$ will be denoted by $HC^*_\mathcal{H}(A,M)$. We define a cyclic $R$-module as
$$C_\bullet(A,M)=\mathcal{Q}_\bullet(A,M) \otimes_\mathcal{H} R,$$
where $R$ is considered as an $\mathcal{H}$-module via the map $s_R\varepsilon_R$. The cyclic module $C(A,M)$ is called the Hopf-cyclic module of the pair $(A, M)$. Its cyclic cohomology will be denoted by $HC^*_\text{Hopf}(A,M)$.
\end{definition}

\subsection{$\mathcal{H}$-categories}
\begin{definition}
A small $k$-linear category $\mathcal{C}$ is said to be an $\mathcal{H}$-category if
\begin{enumerate}
\item[(i)] $\text{Hom}_\mathcal{C}(X,Y)$ is a left $\mathcal{H}$-module for all $X,Y \in Ob(\mathcal{C}),$ inheriting an $L$-$L$ bimodule structure from the $\mathcal{H}$-action i.e. $$l\cdot f \cdot l'=(l \cdot 1_H \cdot l')f=s_L(l)t_L(l')f,$$ 
for $l,l' \in L$ and $f \in \text{Hom}_\mathcal{C}(X,Y)$, and $l \cdot \text{id}_X=\text{id}_X \cdot l$,
\item[(ii)] for any $X \in Ob(\mathcal{C})$ and $h \in \mathcal{H}$, $h(\text{id}_X)=s_L \varepsilon_L(h)\text{id}_X \equiv t_L \varepsilon_L(h)\text{id}_X$,
\item[(iii)] for $h \in \mathcal{H}$ and any pair of composable morphisms $g:X \longrightarrow Y$, $f: Y \longrightarrow Z$, 
$$h(fg)=(h_{(1)}f)(h_{(2)}g).$$
\end{enumerate}
\end{definition}

\begin{definition}
A morphism $f:X \longrightarrow Y$ in an $\mathcal{H}$-category is said to be $\mathcal{H}$-invariant if $h(f)=s_L\varepsilon_L(h)f=t_L\varepsilon_L(h)f$ for all $h \in \mathcal{H}$.
\end{definition}

\begin{lemma}
Let $\mathcal{H}=(\mathcal{H}_L,\mathcal{H}_R,S)$ be a Hopf algebroid, where $S$ is assumed to be bijective, and $A$ be a left $\mathcal{H}$-module algebra. Then ${_\mathcal{H}}{M}{^A}$, the category of all finitely generated left $\mathcal{H}$-equivariant right $A$-modules with the morphisms are that of right $A$-modules, is an $\mathcal{H}$-category. 
\end{lemma}
\begin{proof}
For $X,Y \in Ob(\mathcal{C})$, $\text{Hom}_A(X,Y)$ is a left $\mathcal{H}$-module via $$h(f)(x)=h_{(1)}f(S(h_{(2)})x),$$ for any $x \in X$ and $f \in \text{Hom}_A(X,Y)$. We claim that $h(f)$ is still a morphism of right $A$-modules because
\begin{align*}
h(f)(xa) &= h_{(1)}f(S(h_{(2)})xa)=h_{(1)}f(S(h_{(2)})x)S(h_{(3)})a &\\
&= h_{(1)}f(S(h_{(2)})x)S(h_{(3)}\varepsilon_L(h_{(4)}))a\\
&= h_{(1)}f(S(h_{(2)})x)S(h_{(3)})S(t_L\varepsilon_L(h_{(4)}))a\\
&= h_{(1)}f(S(h_{(2)})x)S(h_{(3)})s_L\varepsilon_L(h_{(4)})a\\
&= h_{(1)}f(S(h_{(2)})x)S(h_{(3)})h_{(4)}a\\
&= h_{(1)}f(S(h_{(2)})x)s_R\varepsilon_R(h_{(3)})a\\
&= h_{(1)}f(S(h_{(2)})x)s_L \varepsilon_L s_R\varepsilon_R(h_{(3)})a\\
&= h_{(1)}f(S(h_{(2)})x)\varepsilon_L (S(h_{(3)}))a\\
&= h_{(1)}f((S(h_{(2)})x) \varepsilon_L (S(h_{(3)})))a\\
&= h_{(1)}f((S(h_{(2)})x)t_L \varepsilon_L (S(h_{(3)})))a\\
&= h_{(1)}f(S(h_{(2)}x) (S(h_{(3)})1_A ))a\\
&= h_{(1)}f(S(h_{(2)})x)a.
\end{align*} 
The fourth equality follows from (\ref{lbimod}) and (\ref{antialg}). The fifth equality follows by the fact that $S$ is an anti-algebra map and $St_L=s_L$, and  sixth one by the property of the module algebra $A$. The seventh equality follows by (\ref{antipode}), and eleventh follows from the fact that $X$ is a right $L$-module via the map $t_L$.
 \end{proof}
 
Let $\mathcal{H}$ and $A$ be as in the previous lemma. Then we have another examples of $\mathcal{H}$-categories: the category ${_\mathcal{H}}{\mathcal{F}}{^A}$ consisting of finitely generated free right $A$-modules (which are left $\mathcal{H}$-equivariant right $A$-modules
automatically), and the category $\ast^A_{\mathcal{H}}$ consisting of only one object $A$ considered as a left $\mathcal{H}$-equivariant right $A$-module.

\begin{definition}
Let $A$ be a left $\mathcal{H}$-module algebra. A projective $A$-module $P$ is called $\mathcal{H}$-equivariantly projective if 
\begin{itemize}
\item[(i)] $P$ is an $\mathcal{H}$-equivariant $A$ module.
\item[(ii)] there exists another $\mathcal{H}$-equivariant $A$ module $Q$ such that $F:=P \oplus Q$ is a free $\mathcal{H}$-equivariant $A$ module.
\item[(iii)] the canonical epimorphism $F \longrightarrow P$ is $\mathcal{H}$-invariant.
\end{itemize}
\end{definition}

The subcategory of ${_\mathcal{H}}{M}{^A}$ which consists of finitely generated $\mathcal{H}$-equivariantly projective right $A$-modules will be denoted by ${_\mathcal{H}}{\mathcal{P}}{^A}$. Note that this category consists of direct summands (both as $A$- and $\mathcal{H}$-modules) of finitely generated free $A$-modules with
their canonical $\mathcal{H}$-module structure.

\begin{definition}[Compare \cite{kk}]
Let $\mathcal{C}$ be an $\mathcal{H}$-category, and let $\xi$ be a full $\mathcal{H}$-subcategory of $\mathcal{C}$. The category $\mathcal{C}$ is called separated over $\xi$ iff for each object $X \in \mathcal{C}$, there exist a finite set of objects $\{X_1, \ldots, X_n\}$ in $\xi$ such that there are $\mathcal{H}$-invariant morphisms $u_i: X \longrightarrow X_i$ and $v_j: X_j \longrightarrow X$ in $\mathcal{C}$ satisfying $\sum_{i=1}^n v_iu_i=id_X$.
\end{definition}

\begin{lemma}\label{sep}
${_\mathcal{H}}{\mathcal{P}}{^A}$ is separated over ${_\mathcal{H}}{\mathcal{F}}{^A}$ which in turn is separated over $\ast^A_{\mathcal{H}}$.
\end{lemma}
\begin{proof}
Let $X$ be a finitely generated $\mathcal{H}$-equivariantly projective module. Therefore, by definition, there exists a finitely generated $\mathcal{H}$-equivariantly free module $A^{\oplus n}$ and $\mathcal{H}$-invariant morphisms  $\pi_X:A^{\oplus n} \longrightarrow X$ and $i_X:X \longrightarrow A^{\oplus n}$ of $A$-modules such that $\pi_X \circ i_X=\text{id}_X$. This proves the first assertion. Next for any finitely generated $\mathcal{H}$-equivariantly free $A$-module $A^{\oplus n}$, there are canonical $\mathcal{H}$-invariant morphisms  $\pi_j:A^{\oplus n} \longrightarrow A$ and $i_j:A \longrightarrow A^{\oplus n}$ satisfying $\sum_j i_j \circ \pi_j=\text{id}_{A^{\oplus n}}$.
\end{proof}

\subsection{Cyclic module for an $\mathcal{H}$-category}

\begin{definition}
Let $\mathcal{H}$ be a Hopf algebroid with a bijective antipode $S$. Let $\mathcal{C}$ be an $\mathcal{H}$-category and $M$ be a stable left-left $\mathcal{H}$-module/comodule. Keeping in mind the fact that $L$ and $R$ are anti-isomorphic $k$-algebras, we define a pseudo para-cyclic right $\mathcal{H}$-module $\mathcal{T_\bullet}(\mathcal{C},M)$ by setting
$$\mathcal{T}_n(\mathcal{C},M)=\bigoplus \text{Hom}_{\mathcal{C}}(X_1,X_0) \otimes_L \cdots \otimes_L \text{Hom}_{\mathcal{C}}(X_n,X_{n-1}) \otimes_L \text{Hom}_{\mathcal{C}}(X_0,X_n) \otimes_L M,$$
where the direct sum runs over all $(X_0,X_1,\ldots,X_n) \in \text{Ob}(\mathcal{C})^{n+1}$. The faces, degeneracies and the cyclic operator, respectively, are defined as

$\begin{array}{lll}
d_i^{n-1}(f_0 \otimes_L \cdots \otimes_L f_n \otimes_L m) &=& \begin{cases} f_0 \otimes_L \cdots \otimes_L f_i f_{i+1} \otimes_L \cdots \otimes_L f_n \otimes_L m \hspace{1.1cm} \text{if}~  0 \leq i \leq n-1\\
(m_{[-1]}f_n)f_0 \otimes_L f_1 \otimes_L \cdots \otimes_L f_{n-1} \otimes_L m_{[0]} \quad \text{if}~ i=n\end{cases}\\\\

s_i^n(f_0 \otimes_L \cdots \otimes_L f_n \otimes_L m) &=& \begin{cases} f_0 \otimes_L \cdots \otimes_L f_i \otimes_L \text{id}_{X_{i+1}} \otimes_L \cdots \otimes_L f_n \otimes_L m \quad \text{if}~  0 \leq i \leq n-1\\
f_0 \otimes_L \cdots \otimes_L f_n \otimes_L \text{id}_{X_0} \otimes_L m \hspace{2.8cm} \text{if}~ i=n\end{cases}\\\\
\tau_n(f_0 \otimes_L \cdots \otimes_L f_n \otimes_L m)&=& m_{[-1]}f_n \otimes_L f_0 \otimes_L \cdots \otimes_L f_{n-1} \otimes_L m_{[0]},
\end{array}$

\vspace{.2cm}
\noindent where $m \mapsto m_{[-1]} \otimes_L m_{[0]}$ is the left $\mathcal{H}_L$-comodule structure of $M$. The right $\mathcal{H}$-action on $\mathcal{T_\bullet}(\mathcal{C},M)$ is given by
$$(f_0 \otimes_L \cdots \otimes_L f_n \otimes_L m)h=S^{-1}(h_{(1)})(f_0) \otimes_L \cdots \otimes_L S^{-1}(h_{(n+1)})(f_n) \otimes_L  S^{-1}(h_{(n+2)})(m),$$
for any $n \geq 0$, $0 \leq i \leq n$, $f_0 \otimes_L \cdots \otimes_L f_n \otimes_L m \in \mathcal{T}_n(\mathcal{C},M)$ and $h \in \mathcal{H}$.

For every $n \geq 0$, we define $q_{n,m}:\mathcal{T}_n(\mathcal{C},M) \longrightarrow  \mathcal{T}_n^m(\mathcal{C},M)$ to be the coequalizer of the morphisms $\rho_n(\tau_n^m \otimes \text{id})$ and $\tau_n^m\rho_n$ for all $m \in \mathbb{N}$, in the category $_R\mathcal{M}_R$ of $R$-$R$-bimodules,  where $\rho_n$ is the right $\mathcal{H}$-module (i.e. right $\mathcal{H}_R$-module) structure morphism on $\mathcal{T}_n$. Then, we define $$\mathcal{Q}^\mathcal{H}_n(\mathcal{C},M):={\text{colim}}_{m \in \mathbb{N}}~ \mathcal{T}_n(\mathcal{C},M) \overset{q_{n,m}}{\longrightarrow}  \mathcal{T}_n^m(\mathcal{C},M),$$
where $q_n:\mathcal{T}_n(\mathcal{C},M) \longrightarrow  \mathcal{Q}^\mathcal{H}_n(\mathcal{C},M) $ is the canonical morphism for any $n \geq 0.$ Applying the construction in \cite[Theorem 4.7]{k} to the category $_R\mathcal{M}_R$, we get the following result. For the sake of completeness, we also provide the proof which is essentially analogous to the proof of Theorem 4.7 in \cite{k}:
\end{definition}

\begin{proposition}
$\mathcal{Q}^\mathcal{H}_\bullet(\mathcal{C},M)$ is a para-cyclic right $\mathcal{H}$-module.
\end{proposition}
\begin{proof} In what follows, we denote $\mathcal{T}_\bullet(\mathcal{C},M)$ simply by $\mathcal{T}_\bullet$  and $\mathcal{Q}^\mathcal{H}_\bullet(\mathcal{C},M)$ by $\mathcal{Q}_\bullet$. Consider the following non-commutative diagram in $_R\mathcal{M}_R$:
\[
\xymatrix{
\mathcal{T}_n \otimes \mathcal{H} \ar[r]^-{\rho_n}  & \mathcal{T}_n \ar[r]^-{q_n} & \mathcal{Q}^\mathcal{H}_n\\
\mathcal{T}_n \otimes \mathcal{H} \ar[r]^-{\rho_n} \ar[u]_{\tau_n^j \otimes \text{id}} & \mathcal{T}_n \ar[r]^-{q_n} \ar[u]_{\tau_n^j} & \mathcal{Q}^\mathcal{H}_n\\
\mathcal{T}_n \otimes \mathcal{H} \ar[r]^-{\rho_n} \ar[u]_{\tau_n^i \otimes \text{id}} & \mathcal{T}_n \ar[u]_{\tau_n^i}\\
}
\]
Since $q_n$ is the coequalizer of the morphisms $\rho_n(\tau_n^i \otimes \text{id})$ and $\tau_n^i\rho_n$ for all $i \in \mathbb{N}$, and 
$$q_n \tau_n^j \tau_n^i \rho_n=q_n\tau_n^j \rho_n (\tau_n^i \otimes \text{id}),$$
we have a restriction of $\tau_n^j$ to $\mathcal{Q}^\mathcal{H}_n$, which will be denoted by $(\tau_n^j)_{\mathcal{H}}$ for any $j \in \mathbb{N}$.

Now for $0 \leq j \leq n$, consider the following non-commutative diagram in $_R\mathcal{M}_R$:
\[
\xymatrix{
\mathcal{T}_{n-1} \otimes \mathcal{H} \ar[r]^-{\rho_{n-1}}  & \mathcal{T}_{n-1} \ar[r]^-{q_{n-1}} & \mathcal{Q}^\mathcal{H}_{n-1}\\
\mathcal{T}_{n} \otimes \mathcal{H} \ar[r]^-{\rho_{n}} \ar[u]_{d_j \otimes \text{id}} & \mathcal{T}_{n} \ar[r]^-{q_{n}} \ar[u]_{d_j} & \mathcal{Q}^\mathcal{H}_{n}\\
\mathcal{T}_{n} \otimes \mathcal{H} \ar[r]^-{\rho_{n}} \ar[u]_{\tau_{n}^i \otimes \text{id}} & \mathcal{T}_{n} \ar[u]_{\tau_{n}^i}\\
}
\]
The upper square commutes whenever $0 \leq j \leq n-1$. Since $d_n=d_0\tau_n$ and $\tau_n$ has a restriction to $\mathcal{Q}^\mathcal{H}_n$, without loss of generality, we can assume that $0 \leq j \leq n-1$. Therefore, using that fact that $\mathcal{T}_\bullet$ is pseudo-para cyclic, it can be proved that
\begin{equation*}
q_{n-1}d_j\tau_n^i \rho_n=q_{n-1}d_j\rho_n(\tau_n^i \otimes \text{id}),
\end{equation*}
for any $i \in \mathbb{N}$. Thus we obtain a unique morphism from $\mathcal{Q}^\mathcal{H}_n \longrightarrow \mathcal{Q}^\mathcal{H}_{n-1}$, which will be denoted by $(d_j)_\mathcal{H}$. Similiarly, One can show the commutativity of relevant diagrams for the degeneracies. This shows that $\mathcal{Q}^\mathcal{H}_\bullet$ is a para-cyclic module in $_R\mathcal{M}_R$.
Now for any $j \in \mathbb{N}$, consider the following non-commutative diagram:
\[
\xymatrix{
\mathcal{T}_n \ar[r]^-{\tau_n^j}  & \mathcal{T}_n \ar[r]^-{q_n} & \mathcal{Q}^\mathcal{H}_n\\
\mathcal{T}_n \otimes \mathcal{H} \ar[r]^-{\tau_n^j \otimes \text{id}} \ar[u]_{\rho_n} & \mathcal{T}_n \otimes \mathcal{H} \ar[r]^-{q_n \otimes \text{id}} \ar[u]_{\rho_n} & \mathcal{Q}^\mathcal{H}_n \otimes \mathcal{H}\\
\mathcal{T}_n \otimes \mathcal{H} \otimes \mathcal{H} \ar[r]^--{\tau_n^j \otimes \text{id} \otimes \text{id}} \ar[u]_{\rho_n \otimes \text{id}} & \mathcal{T}_n \otimes \mathcal{H} \otimes \mathcal{H} \ar[u]_{\rho_n \otimes \text{id}}\\
}
\]
But we have,
\begin{align*}
q_n\rho_n(\rho_n \otimes \text{id})(\tau_n^j \otimes \text{id} \otimes \text{id}) &=q_n\rho_n(\text{id} \otimes \mu_R)(\tau_n^j \otimes \text{id} \otimes \text{id}),\\
&= q_n\rho_n(\tau_n^j \otimes \text{id})(\text{id} \otimes \mu_R),\\
&= q_n \tau_n^j \rho_n (\text{id} \otimes \mu_R),\\
&= q_n\rho_n(\tau_n^j \otimes \text{id})(\rho_n \otimes \text{id}).
\end{align*}
This implies that $q_n\rho_n$ factors through the colimit of the coequalizers of the pairs $(\rho_n \otimes \text{id})(\tau_n^j \otimes \text{id} \otimes\text{id})$ and $(\tau_n^j \otimes \text{id})(\rho_n \otimes \text{id})$ as $j$ runs through $\mathbb{N}$. But this colimit is $\mathcal{Q}^\mathcal{H}_n \otimes \mathcal{H}.$ Thus we get a right $\mathcal{H}$-module structure on $\mathcal{Q}^\mathcal{H}_n$, and hence $\mathcal{Q}^\mathcal{H}_\bullet$ is a para-cyclic module in the category of right $\mathcal{H}$-modules.

Now we need to show that the structure morphisms of the para-cyclic module $\mathcal{Q}^\mathcal{H}_n$ commutes with the action of $\mathcal{H}$. For that we need to check the commutativity of the larger squares in the following diagrams, for any $i \geq 0$ and $0 \leq j \leq n$:
\[
\xymatrix{
\mathcal{T}_{n} \otimes \mathcal{H} \ar[r]^-{d_j \otimes \text{id}}  \ar[d]^-{\rho_n} & \mathcal{T}_{n-1} \otimes \mathcal{H} \ar[d]^-{\rho_{n-1}}\\
\mathcal{T}_{n} \ar[r]^-{d_j} \ar[d]_{q_n} & \mathcal{T}_{n-1} \ar[d]^-{q_{n-1}}\\
\mathcal{Q}^\mathcal{H}_n \ar[r]^-{{(d_j)}_\mathcal{H}} & \mathcal{Q}^\mathcal{H}_{n-1}\\
} \hspace{2cm}
\xymatrix{
\mathcal{T}_{n} \otimes \mathcal{H}  \ar[d]^-{\rho_n} & \mathcal{T}_{n-1} \otimes \mathcal{H} \ar[l]_-{s_j \otimes \text{id}} \ar[d]^-{\rho_{n-1}}\\
\mathcal{T}_{n} \ar[d]_{q_n} & \mathcal{T}_{n-1} \ar[l]_-{s_j}  \ar[d]^-{q_{n-1}}\\
\mathcal{Q}^\mathcal{H}_n  & \mathcal{Q}^\mathcal{H}_{n-1} \ar[l]_-{{(s_j)}_\mathcal{H}}\\
}
\]
\[
\xymatrix{
\mathcal{T}_{n} \otimes \mathcal{H} \ar[r]^-{\tau_n^i \otimes \text{id}}  \ar[d]^-{\rho_n} & \mathcal{T}_{n} \otimes \mathcal{H} \ar[d]^-{\rho_{n}}\\
\mathcal{T}_{n} \ar[r]^-{\tau_n^i} \ar[d]_{q_n} & \mathcal{T}_{n} \ar[d]^-{q_{n}}\\
\mathcal{Q}^\mathcal{H}_n \ar[r]^-{{(\tau_n^i)}_\mathcal{H}} & \mathcal{Q}^\mathcal{H}_n\\
} 
\]
For the first two diagrams, the commutativity of the upper squares follow from the fact that $\mathcal{T}_\bullet$ is pseudo-para cyclic. The bottom squares are already shown to be commutative. Therefore both the diagrams are commutative. In the last diagram, the bottom square commutes, and, since $q_n$ coequalizes $\rho_n(\tau_n^i \otimes \text{id})$ and $\tau_n^i\rho_n$, the larger square also commutes. Thus $\mathcal{Q}^\mathcal{H}_\bullet$ is a para-cyclic right $\mathcal{H}$-module.
\end{proof}

\begin{definition}
For every $n \geq 0$, let $\mathcal{Q}_n(\mathcal{C},M)$ be the equalizer of the pair $(\tau_n^{n+1},\text{id})$ of para-cyclic right $\mathcal{H}$-module morphisms, which implies that $\mathcal{Q}_n(\mathcal{C},M)$ is also a para-cyclic right $\mathcal{H}$-module. Moreover, $\tau_n^{n+1}=\text{id}$ on $\mathcal{Q}_n(\mathcal{C},M)$. Thus $\mathcal{Q}_\bullet(\mathcal{C},M)$ is a cyclic right $\mathcal{H}$-module.
\end{definition}

\begin{remark}\label{oneobjiso}
$\mathcal{Q}(\ast^A_{\mathcal{H}},M)$ is isomorphic to $\mathcal{Q}(A,M)$ (defined in Subsection \ref{EHC}) as cyclic $\mathcal{H}$-modules, where $\ast^A_{\mathcal{H}}$ is the $\mathcal{H}$-category consisting of only one object $A$ considered as a left $\mathcal{H}$-equivariant right $A$-module.
\end{remark}

\begin{example}
Let $\mathcal{H}=(\mathcal{H}_L,\mathcal{H}_R,S)$ be a Hopf algebroid, $\mathcal{S}$ be a category with two objects $0$ and $1$, and two non-identity morphisms $0 \longrightarrow 1$ and $1 \longrightarrow 0$. Let $L\mathcal{S}$ be the $L$-linear category with $\text{Ob}(L\mathcal{S})=\text{Ob}(\mathcal{S})$ and $\text{Hom}_{L\mathcal{S}}(i,j)$ the free $L$-module on the set $\text{Hom}_{\mathcal{S}}(i,j)$ for $i,j \in \{0,1\}$. Then $L\mathcal{S}$ is an $\mathcal{H}$-category via the map $s_L\varepsilon_L$. Thus one can consider the cyclic $\mathcal{H}$-module $CC_\bullet(L\mathcal{S}):=\mathcal{T}_\bullet(L\mathcal{S},L)=\mathcal{Q}_\bullet(L\mathcal{S},L).$
\end{example}

\subsection{Special homotopy and Morita invariance}
Recall from \cite{carthy}, a map $\lambda$ between two cyclic $\mathcal{H}$-modules $X_\bullet$ and $Y_\bullet$ is called semi-cyclic 
if it is a map of graded $\mathcal{H}$-modules that commutes with the face maps and the cyclic operators but not necessarily with the degeneracy operators. Two semi-cyclic module maps $f,g: X_\bullet \longrightarrow Y_\bullet$ of cyclic $\mathcal{H}$-modules $X_\bullet$ and $Y_\bullet$, are called special homotopic if there exists a semi-cyclic $\mathcal{H}$-module 
map $\hbar_\bullet$ from $X_\bullet \otimes CC_\bullet(L\mathcal{S})$ to $Y_\bullet$ such that the following diagram commutes:
\[
\xymatrix{
X_\bullet \ar[r]^-{\varepsilon_0} \ar[dr]_{f} & X_\bullet \otimes CC_\bullet(L\mathcal{S}) \ar[d]^{\hbar_\bullet} & X_\bullet \ar[l]_-{\varepsilon_1} \ar[dl]^{g}\\
 & Y_\bullet
}
\]

\noindent where $\varepsilon_i(\Phi)=\Phi \otimes_L (\text{id}_i, \ldots, \text{id}_i)$ for $i=0,1$.

\begin{remark}\label{homotopy}
We would like to recall from \cite{carthy} the following fact: a special homotopy implies that the two maps are homotopic.
\end{remark}

The following is a generalization of Proposition 2.4.1 from \cite{carthy} for cyclic $\mathcal{H}$-modules.

\begin{proposition}\label{homoeq}
Let $\mathcal{F}, \mathcal{G}:\mathcal{C} \longrightarrow \mathcal{C}'$ be $\mathcal{H}$-equivariant equivalent functors of $\mathcal{H}$-categories. Then $\mathcal{Q}_\bullet(\mathcal{F},M)$ and $\mathcal{Q}_\bullet(\mathcal{G},M)$ are special homotopic for a fixed stable $\mathcal{H}$-module/comodule $M$. 
\end{proposition}
\begin{proof}
Let $h$ be a natural isomorphism from $\mathcal{F}$ to $\mathcal{G}$ of $\mathcal{H}$-equivariant functors i.e. $h$ is $\mathcal{H}$-invariant. We define a morphism $h_\bullet$ from $\mathcal{T}_\bullet(\mathcal{C},M) \otimes CC_\bullet(L\mathcal{S})$ to $\mathcal{T}_\bullet(\mathcal{C'},M)$ of pseudo para-cyclic $\mathcal{H}$-modules
$$h((f_0 \otimes_L \cdots \otimes_L f_n \otimes_L m) \otimes (i_0 \otimes_L \cdots \otimes_L i_n \otimes_L i_0))=\mu(f_0,i_0)\otimes_L \cdots \otimes_L \mu(f_n,i_n) \otimes_L m,$$
where for any $f \in \text{Hom}_\mathcal{C}(X,Y)$ we let

$$ \mu(f,i)=
\begin{cases} \begin{array}{lcc} \mathcal{F}(f)  & \text{if}& i = \text{id}_0,\\
\mathcal{F}(f)h_X^{-1}  & \text{if}& \hspace{.7cm} i= 1 \longrightarrow 0,\\
h_Y\mathcal{F}(f)  & \text{if}& \hspace{.7cm} i= 0 \longrightarrow 1,\\
h_Y\mathcal{F}(f)h_X^{-1}  & \text{if}& i = \text{id}_1.
\end{array}
\end{cases}$$

This induces a semi-cyclic $\mathcal{H}$-module map $\hbar_\bullet$ from $\mathcal{Q}_\bullet(\mathcal{C},M) \otimes CC_\bullet(L\mathcal{S})$ to $\mathcal{Q}_\bullet(\mathcal{C'},M)$, which is also a special homotopy as $\hbar_\bullet \varepsilon_0=\mathcal{Q}_\bullet(\mathcal{F},M)$ and $\hbar_\bullet \varepsilon_1=\mathcal{Q}_\bullet(\mathcal{G},M)$.
\end{proof}

\begin{corollary}\label{homoequi}
If two $\mathcal{H}$-categories $\mathcal{C}$ and $\mathcal{C'}$ are $\mathcal{H}$-equivariantly equivalent, then $\mathcal{Q}_\bullet(\mathcal{C},M)$ and $\mathcal{Q}_\bullet(\mathcal{C'},M)$ are $\mathcal{H}$-equivariantly homotopy equivalent, for a fixed stable $\mathcal{H}$-module/comodule $M$. 
\end{corollary}
\begin{proof}
Let $\mathcal{F}:\mathcal{C} \longrightarrow \mathcal{C'}$ and $\mathcal{G}:\mathcal{C'} \longrightarrow \mathcal{C}$ be the $\mathcal{H}$-equivariant functors such that $\mathcal{G}\mathcal{F}\simeq \text{id}_\mathcal{C}$ and $\mathcal{F}\mathcal{G} \simeq \text{id}_\mathcal{C'}$ $\mathcal{H}$-equivariantly which means that the natural transformations $\hbar:\mathcal{F} \longrightarrow \mathcal{G}$ and $\hbar':\mathcal{G} \longrightarrow \mathcal{F}$ are $\mathcal{H}$-invariant. Then the result follows by applying Proposition \ref{homoeq} to the functors $\mathcal{F}\mathcal{G}$ and $\text{id}_\mathcal{C}'$, and the functors $\mathcal{G}\mathcal{F}$ and $\text{id}_\mathcal{C}$, and Remark \ref{homotopy}.
\end{proof}

\begin{proposition}\label{inclu}
Let $\mathcal{C}$ be an $\mathcal{H}$-category separated over a $\mathcal{H}$-subcategory $\xi$. Then
the natural inclusion of functors $\xi \longrightarrow \mathcal{C}$ induces a homotopy equivalence of cyclic $\mathcal{H}$-modules $\mathcal{Q}_\bullet(\mathcal{\xi},M) \longrightarrow \mathcal{Q}_\bullet(\mathcal{C},M)$ for any stable $\mathcal{H}$-module/comodule M.
\end{proposition}
\begin{proof}
Since $\mathcal{C}$ is $\mathcal{H}$-separated over $\xi$, we have for each $X \in \text{Ob}(\mathcal{C})$, a finite set of objects $\{X_1,\ldots, X_{n(X)}\}$ in $\xi$ and $\mathcal{H}$-invariant morphisms $u_i: X \longrightarrow X_i$ and $v_j: X_j \longrightarrow X$ in $\mathcal{C}$ satisfying $\sum_{i=1}^{n(X)} v_iu_i=id_X.$ We define a semi-cyclic $\mathcal{H}$-module map $\mathcal{E}_\bullet:\mathcal{Q}_\bullet(\mathcal{C},M) \longrightarrow \mathcal{Q}_\bullet(\xi,M)$ as
$$\mathcal{E}_n(f_0 \otimes_L \cdots \otimes_L f_n \otimes_L m)=\sum u_{i_0}(X_0)f_0v_{i_1}(X_1) \otimes_L u_{i_{n-1}}(X_{n-1})f_{n-1}v_{i_n}(X_n) \otimes_L u_{i_n}(X_n)f_nv_{i_0}(X_0) \otimes_L m$$
for any $f_0 \otimes_L \cdots \otimes_L f_n \otimes_L m \in \mathcal{Q}_\bullet(\mathcal{C},M).$ 

Now we construct a special homotopy $\hbar_\bullet$ between the maps $i_\bullet \mathcal{E}_\bullet$ and $\text{id}_\bullet$ as follows:
$$\hbar((f_0 \otimes_L \cdots \otimes_L f_n \otimes_L m) \otimes (i_0 \otimes_L \cdots \otimes_L i_n \otimes_L i_0))=\mu(f_0,i_0)\otimes_L \cdots \otimes_L \mu(f_n,i_n) \otimes_L m,$$ where 
$ \mu(f,i)=
\begin{cases} \begin{array}{lcc}f  & \text{if}~& i = \text{id}_0,\\
\sum_j u_j(Y)f  & \text{if}~& i= 1 \longrightarrow 0,\\
\sum_j fv_j(X)  & \text{if}~& i= 0 \longrightarrow 1,\\
\sum_{i,j}u_j(Y)fv_j(X)  & \text{if}~& i = \text{id}_1.
\end{array}
\end{cases}$

One can see that $\hbar_\bullet\varepsilon_0=\text{id}_\bullet$ and $\hbar_\bullet\varepsilon_1=i_\bullet \mathcal{E}_\bullet$.
\end{proof}

\begin{definition}[Compare \cite{kk}]
Two $\mathcal{H}$-module algebras $A$ and $A'$ are said to be $\mathcal{H}$-equivariantly
Morita equivalent if the categories ${_\mathcal{H}}{\mathcal{P}}{^A}$ and ${_\mathcal{H}}{\mathcal{P}}{^{A'}}$ are $\mathcal{H}$-equivariantly equivalent.
\end{definition}

\begin{proposition}
Let $\mathcal{F}:{_\mathcal{H}}{\mathcal{P}}{^A} \longrightarrow {_\mathcal{H}}{\mathcal{P}}{^{A'}}$ be an $\mathcal{H}$-equivariant functor commuting with direct sums. Then there exists an $A$-$A'$-bimodule $P$ such that the functor $\mathcal{F}$ is naturally isomorphic to the functor $(\cdot) \otimes_A P$ as $\mathcal{H}$-equivariant functors.
\end{proposition}
\begin{proof}
We first claim that the statement holds for the subcategory of ${_\mathcal{H}}{\mathcal{P}}{^A}$ which consists of $\mathcal{H}$-equivariant modules isomorphic to the modules of the form $A^{\oplus n}$ for some $n \geq 1$. Let $P:=\mathcal{F}(A)$. Then, by definition, $P$ is a right $A'$-module. For $a \in A$, the $L$-linear morphism $\mathcal{F}(L_a):\mathcal{F}(A) \longrightarrow \mathcal{F}(A)$ defined by left multiplication by $a$, is a right $A'$ module morphism. Thus $P$ is a left $A$-module. 

Consider the $\mathcal{H}$-equivariant $A$-module $A^{\oplus n} \otimes_A P$, where the $\mathcal{H}$-module structure on the tensor product is diagonal. We define an isomorphism of $\mathcal{H}$-equivariant $A$-modules $\alpha_n:A^{\oplus n} \otimes_A P \longrightarrow P^{\oplus n}$ as
$$\alpha_n((a_1,\ldots,a_n) \otimes_A x)=(a_1x,\ldots,a_nx),$$ for any $n \geq 1$, $a_i \in A$ and $x \in P$. Thus we get an isomorphism 
\[
\xymatrix{
A^{\oplus n} \otimes_A P \ar[r]^-{\alpha_n} & P^{\oplus n} \ar[r]^-{\beta_n} & F(A^{\oplus n}),
}
\]
where $\beta_n$ is the canonical isomorphism of $\mathcal{H}$-equivariant $A$-modules. This proves the claim. 

Now, let $X \in \text{Ob}({_\mathcal{H}}{\mathcal{P}}{^A})$. Then, by definition, there exist $X' \in \text{Ob}({_\mathcal{H}}{\mathcal{P}}{^A})$ such that $X \oplus X' \cong A^{\oplus n}$ as $\mathcal{H}$-equivariant $A$-modules for some $n \geq 1$, and $\mathcal{H}$-invariant morphisms $i_X: X \longrightarrow A^{\oplus n}$ and $\pi_X: A^{\oplus n} \longrightarrow X$ satisfying $\pi_Xi_X=\text{id}_X$. We define a natural transformation $\eta:(\cdot) \otimes_A P\longrightarrow \mathcal{F}$ by the composition
\[
\xymatrix{
X \otimes_A P \ar[r]^-{i_X \otimes_A P} & (X \oplus X') \otimes_A P \ar[r]^-{\beta_n \alpha_n}& \mathcal{F}(X \oplus X') \ar[r]^-{\mathcal{F}(\pi_X)}& \mathcal{F}(X).
}
\]
The map $\eta$ is, in fact, an isomorphism with inverse given by $(\pi_X \otimes_A P) \alpha_n^{-1}\beta_n^{-1}\mathcal{F}(i_X)$.
\end{proof}

\begin{theorem}
Let $A$ and $A'$ be unital $\mathcal{H}$-module algebras which are $\mathcal{H}$-equivariantly Morita equivalent and $M$ be an stable $\mathcal{H}$-module/comodule. Then the $\mathcal{H}$-equivariant and Hopf-cyclic cohomologies of the pairs $(A,M)$ and $(A',M)$ are isomorphic.
\end{theorem}
\begin{proof}
The cyclic $\mathcal{H}$-modules $\mathcal{Q}_\bullet(_{\mathcal{H}}{\mathcal{P}}{^A},M)$ and $\mathcal{Q}_\bullet(_{\mathcal{H}}{\mathcal{P}}{^{A'}},M)$ are homotopy equivalent by Corollary \ref{homoequi}. Since $\mathcal{Q}_\bullet(_{\mathcal{H}}{\mathcal{P}}{^A},M)$ is separated over $\ast^A_{\mathcal{H}}$ by Lemma \ref{sep}, the inclusion functor $\ast^A_{\mathcal{H}} \hookrightarrow \mathcal{Q}_\bullet(_{\mathcal{H}}{\mathcal{P}}{^A},M)$ induces a homotopy equivalence between $\mathcal{Q}_\bullet(\ast^A_{\mathcal{H}},M)$ and $\mathcal{Q}_\bullet(_{\mathcal{H}}{\mathcal{P}}{^A},M)$ by Proposition \ref{inclu}. The result follows by Remark \ref{oneobjiso}.
\end{proof}

\section*{Acknowledgements}
This work is supported by the National Postdoctoral Fellowship PDF/2017/000229 from SERB, Department of Science \& Technology, Government of India. The author thanks Abhishek Banerjee for several stimulating discussions. The author also thanks Atabey Kaygun for helpful conversations.

\end{document}